%% file: Rabbit-arxiv.tex
\documentclass[a4paper,11pt]{article}
\usepackage{amsmath,amsthm,amsfonts,amssymb}

\usepackage{graphicx,color}
\usepackage{boxedminipage}
\usepackage{todonotes}
\usepackage{comment}
\newtheorem{theorem}{Theorem}

\newtheorem{proposition}{Proposition}
\newtheorem{corollary}{Corollary}
\newtheorem{lemma}{Lemma}
\newtheorem{claim}{Claim}

\newcommand{\hunt}{h}
\newcommand{\Oh}{O}
\def\cqedsymbol{\ifmmode$\lrcorner$\else{\unskip\nobreak\hfil
\penalty50\hskip1em\null\nobreak\hfil$\lrcorner$
\parfillskip=0pt\finalhyphendemerits=0\endgraf}\fi} 

\newcommand{\cqed}{\renewcommand{\qed}{\cqedsymbol}}

\newcommand{\pw}{{\mathbf{pw}}}

\begin{document}

\title{How to Hunt an Invisible Rabbit on a Graph
\footnote{The research leading to these results has received funding from the European Research Council under the European Union's Seventh Framework Programme (FP/2007-2013) / ERC Grant Agreement n.~267959. M.~Pilipczuk is currently holding a post-doc position at Warsaw Center of Mathematics and Computer Science; However, this work was done when M.~Pilipczuk was affiliated with the University of Bergen.}
 }

\author{
Tatjana  V. Abramovskaya\thanks{Department of Operations Research, St.-Petersburg State University,  Universitetsky prospekt 28, Stary Peterhof, 198504, St-Petersburg, Russia. E-mail: {\tt tanya.abramovskaya@gmail.com}}
\and
Fedor V. Fomin\thanks{Department of Informatics, University of Bergen, PB 7803, 5020 Bergen, Norway. E-mails: {\tt{\{fedor.fomin,petr.golovach\}@ii.uib.no}}}
\addtocounter{footnote}{-1}
\and 
Petr A. Golovach\footnotemark
\and
Micha{\l} Pilipczuk\thanks{Institute of Informatics, University of Warsaw, ul. Banacha 2, 02-097 Warsaw, Poland. E-mail: {\tt{michal.pilipczuk@mimuw.edu.pl}}}
}

\date{}

\maketitle

\begin{abstract}
We investigate Hunters \& Rabbit game,  where a set of hunters tries to catch an invisible  rabbit that slides along the edges of a graph. 
We show that the minimum number of hunters required to win on an $(n\times m)$-grid is 
  $\lfloor \frac{\min\{n,m\}}{2}\rfloor+1$.  
We also show that the extremal value of this number on $n$-vertex trees is between $\Omega(\log n/\log \log n)$ and $\Oh(\log n)$.
\end{abstract}

\section{Introduction}

Our work originated from the following game puzzle.
 Hunter wants to shoot Rabbit who is hiding behind one of the three bushes growing in a row. Hunter does not see Rabbit,  so he select one of the bushes and shoots at it. 
 If Rabbit is behind the selected bush, then Hunter wins. Otherwise Rabbit, scared by the shot, jumps to one of the adjacent bushes. 
 As Rabbit is infinitely fast, Hunter sees neither Rabbit's old nor new bush and has to select where to shoot again. 
 Can Hunter always win in this game? 
 
 Of course, the answer is \emph{yes}: Hunter has to shoot twice at the middle  bush.
 If he misses the first time, it means that Rabbit was hiding either behind  
  the leftmost or the rightmost bush. In both cases, the only adjacent bush where Rabbit can jump after the first shot is the middle one, thus the second shot at the middle bush finishes the game. 
 A natural question is what happens if we have four bushes, and more generally, $n\geq 3$ bushes growing in a row? After a bit of thinking, the answer here is \emph{yes} as well. This time Hunter wins by  shooting consequently at the bushes $2,\ldots, n-1$ when $n$ is odd and at the bushes $2,\ldots, n$ when $n$ is even, and then repeating the same sequence of shots again.

In a slightly different situation, when   bushes grow around a circle, say we have three bushes and Rabbit can jump from any of them to any of them, then Hunter cannot guarantee the success anymore. In this situation we need the second hunter and this brings us to the following setting. We consider  Hunters \& Rabbit game with  
  two players, Hunter and Rabbit, playing  on an undirected graph.  Hunter player has a team of hunters who attempt to shoot the rabbit. At the beginning of the game, Rabbit  player selects a vertex and occupies it. Then the players take turns starting with Hunter player. At
every round of the game each of the hunters selects some vertex of the graph and the hunters shoot simultaneously at their respective aims. If the rabbit is not in a vertex that is hit by a shot, it jumps to an adjacent vertex. The rabbit is invisible to the hunters, but since we are interested in the guaranteed success of the hunters, we can assume that rabbit has a complete knowledge about all shots that the hunters plan.
Hunter player wins if at some step of the game he succeeds to shoot the
rabbit, and Rabbit player wins if the rabbit can avoid this situations forever.
For a given graph $G$,
we are interested in the minimum number of hunters sufficient to win in the Hunters \& Rabbit game on $G$, for any strategy chosen by the rabbit player. We call this parameter the \emph{hunter number} of a graph, and denote it by $h(G)$.

\paragraph{Related work} 
Britnell and Wildon studied the case with one hunter in~\cite{BritnellW13}. They characterized the graphs for which one hunter (the prince in their terminology) can find the rabbit (the princess).  

Hunters \& Rabbit game is  closely related to several pursuit-evasion and search games on graphs, see \cite{FominT08} for further references.  In pursuit-evasion games a team of cops is trying to catch a robber located on the vertices of the graph. In cops-robbers terminology, Hunters \& Rabbit is the Cops \& Robber game, where the set of cops on helicopters (i.e. allowed to jump to any vertex) is trying to catch an invisible robber. The robber moves only to adjacent vertices and is forced to move every time the cops are in the air.
 
In particular, the classical Cops \& Robbers games introduced 
independently by Winkler and Nowakowski~\cite{NowakowskiW83} and by Quilliot~\cite{Quilliot85} (see also  the book by Bonato and Nowakowski~\cite{BonatoN11} for the detailed introduction to the field), is the game where robber is visible, and both players, Cops and Robber, can move their  men only to adjacent vertices. The variant of the game where the robber is invisible introduced by To{\v{s}}i{\'c}~\cite{Tosic85} and the variant where the cops use predefined paths as theirs search moves was introduced by Brass et al.~\cite{Brass09}.
Another related search game, node search, was introduced by Kirousis and Papadimitriou
in~\cite{KirousisP85-In,KirousisP86-Se}. Here cops can fly, i.e. move to any vertex they wish, the robber is invisible and very fast, i.e. can go to any vertex connected to his current location by a path containing no cops. 
Thus, Hunters \& Rabbit can be seen as a variant of To{\v{s}}i{\'c}'s game where cops have more power or as a variant of Kirousis-Papadimitriou's game, where the robber is more restricted. One more significant difference with mentioned games is that  in most versions of Cops \& Robbers games the robber is not forced to move at every step of the game, while in our setting the rabbit cannot stay at the same vertex for to consecutive steps.

A randomized game called Hunter vs. Rabbit was considered by Adler et al.~\cite{AdlerRSSV03}; here, the hunter is allowed to move only along edges of the graph while there are no constrains on rabbit's moves.

\paragraph{Our results and organization of the paper} 
The remaining part of this paper is organized as follows. We give basic definitions and preliminary results in Section~\ref{sec:defs}. In Section~\ref{sec:grid}, we prove our first main result, namely that for an $(n\times m)$-grid $G$, it holds that $h(G)=\lfloor \frac{\min\{n,m\}}{2}\rfloor+1$. This result is based on a new isoperimetric theorem that we find interesting on its own. In Section~\ref{sec:tree}, we provide bounds on the hunting number of trees, which is the second contribution of the paper. We show that the hunting number of an $n$-vertex tree is always $\Oh(\log n)$, but there are trees where it can be as large as $\Omega(\log n/\log \log n)$. We conclude with open problems in Section~\ref{sec:conclusions}.

\section{Basic definitions and preliminaries}\label{sec:defs}

We consider finite undirected graphs without loops or multiple
edges. The vertex set of a (directed) graph $G$ is denoted by $V(G)$, 
the edge set is denoted by $E(G)$.
For a set of vertices $S\subseteq V(G)$,
$G[S]$ denotes the subgraph of $G$ induced by $S$, and by $G-S$ we denote the graph obtained from $G$ by the removal of all the vertices of $S$, i.e., the subgraph of $G$ induced by $V(G)\setminus S$. 
Let $G$ be an undirected graph. 
For a vertex $v$, we denote by $N_G(v)$ its
\emph{(open) neighborhood}, that is, the set of vertices that are
adjacent to $v$. 
The \emph{closed neighborhood} of a vertex $v$ is $N_G[v]=N_G(v)\cup\{v\}$
and the \emph{degree} of $v$ is denoted by $d_G(v)=|N_G(v)|$.
For $S\subseteq V(G)$, we denote $N_G(S)=\cup_{v\in S}N_G(v)\setminus S$ and $\delta_G(S)=|N_G(S)|$. We omit the graph in the subscript whenever it is clear from the context.
For positive integers $n$ and $m$, the \emph{$(n\times m)$-grid} is the graph with the vertex set $\{(x,y)|1\leq x\leq n, 1\leq y\leq m\}$ where $x$ and $y$ are integers, such that two vertices $(x,y)$ and $(x',y')$ are adjacent if and only if $|x-x'|+|y-y'|=1$.

Consider the Hunters \& Rabbit game on a graph $G$. Suppose that the Hunter player has $k$ hunters. A \emph{hunters' strategy} is a (possible infinite) sequence $\mathcal {H}=(H_1,H_2,\ldots)$ where $H_i\subseteq  V(G)$ and  $|H_i|\leq k$ for $i\in\{1,2,\ldots\}$; the hunters shoot at each vertex of $H_i$ at the $i$-th round of the game. Respectively, a \emph{rabbit's strategy}  is a sequence $\mathcal{R}=(r_0,r_1,\ldots)$ of vertices of $G$ such that $r_i$ is adjacent to $r_{i-1}$ for $i\geq 1$; $r_0$ is an initial position of the rabbit, and it jumps from $r_{i-1}$ to $r_i$ after the $i$-th shot of the hunters. 
For a set of vertices $S\subseteq V(G)$, a strategy $\mathcal{H}$ is a \emph{winning hunters' strategy with respect to $S$} if for any rabbit's strategy $\mathcal{R}$ such that $r_0\in S$, there is $i\geq 1$ such that $r_{i-1}\in H_i$;  $\mathcal{H}$ is a \emph{winning hunters' strategy} if it is a winning hunters' strategy with respect to $V(G)$.
Therefore, the hunter number $h(G)$ is the minimum $k$ such that there is a winning hunters' strategy for $k$ hunters. We also say that a  rabbit's strategy $\mathcal{R}$ is a \emph{winning rabbit's strategy against a hunters' strategy $\mathcal{H}$} if $r_{i-1}\notin H_i$ for all $i\geq 1$.

As it is common for pursuit-evasion games with invisible fugitives, it is convenient to keep track of vertices that can or, respectively, cannot be occupied by the rabbit. Let $\mathcal {H}=(H_1,H_2,\ldots)$ be a hunters' strategy. For $S\subseteq V(G)$, a vertex $v$ is \emph{contaminated with respect to $S$} after $i$-th shot if there is a rabbit's strategy $\mathcal{R}=(r_0,r_1,\ldots)$ such that $r_0\in S$, $v=r_i$ and for any $j\in\{1,\ldots,i\}$, $r_{j-1}\notin H_j$. Otherwise, we say that $v$ is \emph{clear with respect to $S$}. If $S=V(G)$, we simply say that $v$ is \emph{contaminated} or \emph{clear}. It is easy to see that if $X$ is a set of vertices contaminated at moment $i-1$ w.r.t. some $S$, then the set of vertices contaminated at moment $i$ will be exactly $\Phi(X,H_i)=N(X\setminus H_i)$.

In our proofs we will be using the fact that we can always restrict our attention to finite strategies. 

\begin{proposition}\label{prop:back}
If $k$ hunters have a winning strategy on an $n$-vertex graph $G$ with respect to $S\subseteq V(G)$, then they have a  winning strategy of length at most $2^n$.
\end{proposition}

\begin{proof}
Consider the auxiliary arena graph, which is a directed  graph $\mathcal{G}$ with the set of vertices $2^{V(G)}$ such that for any distinct $X,Y\subseteq V(G)$, $\mathcal{G}$ has the arc $(X,Y)$ if and only if there exists a set $H\subseteq V(G)$ of size at most $k$ such that $Y=\Phi(X,H)$. The graph $\mathcal{G}$ has $2^n$ vertices and at most $\binom{n}{k}\cdot 2^n$ arcs. It is easy to observe that 
$k$ hunters have a winning strategy on $G$ if and only if $\mathcal{G}$ has a directed walk that leads from $S$ and $\emptyset$: such paths corresponds to Hunter's strategies, while the traversed vertices of $\mathcal{G}$ keep track of the set of contaminated vertices. Moreover, if such a walk exists, then there is also a directed simple path from $S$ to $\emptyset$, which corresponds to a winning strategy of length at most $2^n$. 
\end{proof}

It is straightforward to observe that the hunter number is closed under taking subgraphs. 

\begin{proposition}\label{prop:subgr}
If $G_1$ is a subgraph of $G_2$, then $h(G_1)\leq h(G_2)$.
\end{proposition}

We conclude the section by the following property of Hunters \& Rabbit on bipartite graphs.

\begin{lemma}\label{lem:bip}
Let $G$ be a bipartite graph and let $(V_1,V_2)$ be a bipartition of $V(G)$. Then $k$ hunters have a winning strategy on $G$ if and only if $k$ hunters have a winning strategy with respect to $V_1$.
\end{lemma}

\begin{proof}
Clearly, if  $k$ hunters have a winning strategy $\mathcal{H}$ on $G$, then $\mathcal{H}$ is a winning strategy with respect to $V_1$. Let $\mathcal{H}$ be a winning strategy on $G$ with respect to $V_1$. By Proposition~\ref{prop:back}, we can assume without loss of generality that $\mathcal{H}=(H_1,\ldots,H_{\ell})$ is finite. Moreover, we assume that $\ell$ is odd; otherwise, we just consider  
$\mathcal{H}=(H_1,\ldots,H_{\ell},H_{\ell})$. Let $\mathcal{H}'$ be the strategy obtained by the concatenation of two copies of the sequence $\mathcal{H}$. We claim that $\mathcal{H}'$ is a winning strategy. To see it consider an arbitrary rabbit's strategy $\mathcal{R}=(r_0,r_1,\ldots)$. If $r_0\in V_1$, then
there is $i\in\{1,\ldots,\ell\}$ such that $r_{i-1}\in H_i$ because $\mathcal{H}$ is a winning hunters' strategy with respect to $V_1$. Suppose then that $r_0\in V_2$. If $r_{i-1}\notin H_i$ for $i\in\{1,\ldots,\ell\}$, then $r_{\ell}\in V_1$ because $\ell$ is odd. Then, there is $j\in\{1,\ldots,\ell\}$ such that $r_{\ell+j-1}\in H_j$.
As in the rounds $\ell+1,\ldots,2\ell$ the hunters repeat $\mathcal{H}$, we have that the hunters shoot the rabbit in the $(\ell+j)$-th round for some $j\in\{1,\ldots,\ell\}$.
\end{proof}

\section{Hunting rabbit on a grid}\label{sec:grid}
In this section we compute the hunter number of an $(n\times m)$-grid. Throughout this section we assume that $n\leq m$. Recall that an $(n,m)$-grid has the vertex set $\{(x,y)\, |\, 1\leq x\leq n, 1\leq y\leq m\}$ and two vertices $(x,y)$ and $(x',y')$ are adjacent if and only if $|x-x'|+|y-y'|=1$. 
For a vertex $(i,j)$, we say that $i$ is the \emph{$x$-coordinate} 
and $j$ is the \emph{$y$-coordinate} of $(i,j)$. 
Clearly, grids are bipartite graphs, and we assume in this section that $(V_1,V_2)$, where $V_1=\{(x,y)\, |\, x+y\text{ is even}\}$ and 
$V_2=\{(x,y)\, |\, x+y\text{ is odd}\}$, is the bipartition of the vertex set of a grid.

\subsection{Isoperimetrical properties of grids}\label{sec:iso}
We need some isoperimetrical properties of subsets of $V_1$ for square grids. Let $G$ be an $(n\times n)$-grid. 

For $i\in\{2,\ldots,2n\}$, let $U_i=\{(x,y)\in V(G)\, |\, x+y=i\}$ (see Fig.~\ref{fig:UW})  and
$$s(i)=|U_i|=
\begin{cases}
i-1&\mbox{if }i\leq n+1,\\
2n-i+1&\mbox{if }i>n+1.
\end{cases}$$
It is assumed that $U_i=\emptyset$ if $i\leq 1$ or $i>2n$. 
We denote the vertices of $U_i$ by $u_1^i,\ldots,u_{s(i)}^i$ and assume that they are ordered by the increase of their $y$-coordinate. For $i\in\{1-n,\ldots,n-1\}$,  
$W_i=\{(x,y)\in V(G)\, |\, x-y=i\}$ (see Fig.~\ref{fig:UW}) and $t(i)=|W_i|=n-|i|$; $W_i=\emptyset$ if $i\leq -n$ or $i\geq n$.
Let $W_i=\{w_1^i,\ldots,w_{t(i)}^i\}$ and assume that the vertices are ordered by the increase of their $y$-coordinate.
Notice that $V_1=U_2\cup U_4\cup\ldots\cup U_{2n}=\ldots\cup W_{-2}\cup W_{0}\cup W_2\cup\ldots$ and $V_2=U_3\cup U_5\cup\ldots\cup U_{2n-1}=\ldots\cup W_{-1}\cup W_{1}\cup \ldots $.

\begin{figure}[ht]
\centering\scalebox{0.75}{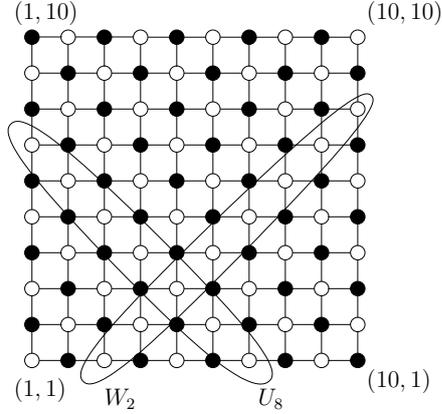}
\caption{Sets $U_8$ and $W_2$ in a $(10\times 10)$-grid; the vertices of $V_1$ are depicted as white bullets and the vertices of $V_2$ are black.
\label{fig:UW}}
\end{figure}

Let $Q\subseteq V_1$. We say that $Q'$ is obtained from $Q$ by the \emph{down-right shifting} if it is constructed as follows: for each even integer $i\in\{2\ldots 2n\}$, all the $r=|U_i\cap Q|$ vertices of $U_i\cap Q$ are replaced by $u_1^i,\ldots,u_r^i$ (see Fig.~\ref{fig:down}). Respectively, 
$Q'$ is obtained from $Q$ by the \emph{down-left shifting} if for each even integer $i\in\{1-n,\ldots, n-1\}$, all the $r=|W_i\cap Q|$ vertices of $W_i\cap Q$ are replaced  by $w_1^i,\ldots,w_r^i$.

\begin{figure}[ht]
\centering\scalebox{0.75}{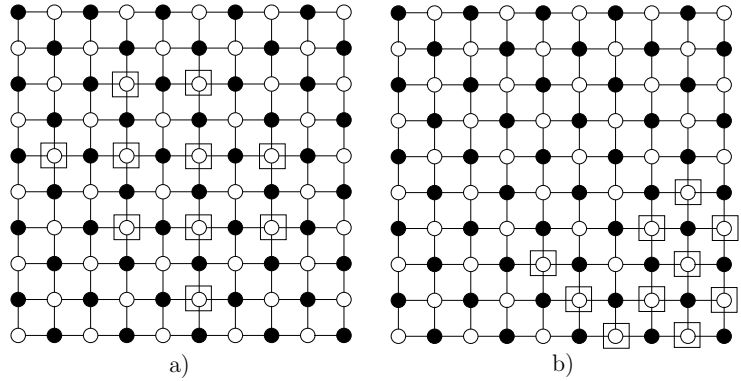}
\caption{An example of the down-right shifting; the vertices of $Q$ (in a) and $Q'$ (in b) are marked by square frames.
\label{fig:down}}
\end{figure}

\begin{lemma}\label{lem:down}
If $Q'$ is obtained from $Q\subseteq V_1$ by the down-right (respectively, down-left) shifting, then $\delta(Q')\leq \delta(Q)$.
\end{lemma}

\begin{proof}
We prove the lemma for the down-right shifting. The proof for the down-left shifting uses symmetric arguments. 

For an odd integer $i\in \{3,5,\ldots,2n-1\}$, let us define the following numbers:
\begin{itemize}
\item $c_i=|U_i\cap N(Q)|$, $c_i^-=|U_i\cap N(Q\cap U_{i-1})|$ and $c_i^+=|U_i\cap N(Q\cap U_{i+1})|$;
\item $d_i=|U_i\cap N(Q')|$, $d_i^-=|U_i\cap N(Q'\cap U_{i-1})|$ and $d_i^+=|U_i\cap N(Q'\cap U_{i+1})|$.
\end{itemize}
By the construction of $Q'$ it is straightforward to verify that $c_i^-\geq d_i^-$ and $c_i^+\geq d_i^+$. Since all elements of $Q$ that neighbor a vertex of $U_i$ reside either in $U_{i-1}$ or in $U_{i+1}$, we have that $c_i\geq \max(c_i^-,c_i^+)$. However, since vertices of $Q'\cap U_{i-1}$ are exactly the $|Q'\cap U_{i-1}|$ vertices of $U_{i-1}$ that have the smallest $y$-coordinate, and the same also holds for $Q'\cap U_{i+1}$, then it is easy to see that $d_i=\max(d_i^-,d_i^+)$. Hence, we obtain that
\begin{eqnarray*}
\delta(Q) & = & \sum_{j=1}^{n-1} c_{2j+1} \geq \sum_{j=1}^{n-1} \max(c_{2j+1}^-,c_{2j+1}^+)\\
& \geq & \sum_{j=1}^{n-1} \max(d_{2j+1}^-,d_{2j+1}^+) = \sum_{j=1}^{n-1} d_{2j+1} = \delta(Q').
\end{eqnarray*}
\end{proof}

\begin{comment}To prove the claim, it is sufficient to make the following straightforward observations:
\begin{itemize}
\item for any even $i\in\{2,\ldots,2n\}$, if $r=|U_i\cap Q|$, then $|N_G(Q\cap U_{i})|\geq |N_G(\{u_1^i,\ldots,u_r^i\})|$;
\item for any even $i\in\{2,\ldots,2n\}$, if $r=|U_i\cap Q|$, then $|N_G(Q)\cap U_{i-1}|\geq |N_G(\{u_1^i,\ldots,u_r^i\})\cap U_{i-1}|$ and  $|N_G(Q)\cap U_{i+1}|\geq |N_G(\{u_1^i,\ldots,u_r^i\})\cap U_{i+1}|$;
\item for any even $i\in\{2,\ldots,2n-2\}$ and any positive integers $r,r'$ such that $r\leq s(i)$ and $r'\leq s(i+1)$ ,   $N_G(\{u_1^i,\ldots,u_r^i\})\cap U_{i+1}\subseteq N_G(\{u_1^{i+2},\ldots,u_{r'}^{i+2}\})\cap U_{i+1}$ or
$N_G(\{u_1^{i+2},\ldots,u_{r'}^{i+2}\})\cap U_{i+1}\subseteq N_G(\{u_1^i,\ldots,u_r^i\})\cap U_{i+1}$. 
\end{itemize}
Then, the claim follows immediately.
\end{comment}

Thus, we already have two operations that preserve the cardinality of a set $Q$ while not incresing $\delta(Q)$: down-left and down-right shifting. We may now inspect sets $Q\subseteq V_1$ that are invariant with respect to both these operations, and it is easy to see that these are exactly sets conforming to the following definition. We say that $Q\subseteq V_1$ is a \emph{pyramidal} set if for any $(x,y)\in Q$ such that $y\geq 2$, $(x-1,y-1)\in Q$ if $x\geq 2$ and $(x+1,y-1)\in Q$ if $x\leq n-1$. 

For $i\in\{1,\ldots,n\}$, let $R_i=\{(x,y)|1\leq x\leq n, y=i\}$, $X_i=R_i\cap V_1$ and $\overline{X}_i=R_i\cap V_2$. Let also 
$$\ell(i)=|X_i|=
\begin{cases}
\lfloor n/2\rfloor &\mbox{if }i\mbox{ is even},\\
\lceil n/2 \rceil &\mbox{if }i\mbox{ is odd}.
\end{cases}$$
Denote by $x_1^i,\ldots,x_{\ell(i)}^i$ the vertices of $X_i$ and assume that they are 
are ordered by the increase of their $x$-coordinate.

\begin{lemma}\label{lem:pyr-formula}
Suppose $Q\subseteq V_1$ is a pyramidal set. Then
$$\delta(Q)=|Q|-|Q\cap X_n|+|N(Q\cap X_1)\cap \overline{X}_1|.$$
\end{lemma}
\begin{proof}
Take any $(x,y)\in N(Q)$ such that $y\geq 2$. As $(x,y)\in N(Q)$, then one of neighboring four vertices of $G$ belongs to $Q$, and due to $Q$ being pyramidal we have that $(x,y-1)\in Q$. Let us construct a matching $M$ between vertices of $N(Q)$ and vertices of $Q$ that matches every vertex $(x,y)\in N(Q)$ with $y\geq 2$ with vertex $(x,y-1)\in Q$. Then, on the side of $N(Q)$ the only unmatched vertices are the vertices of $N(Q)\cap R_1$, and from the fact that $Q$ is pyramidal it follows that these are these are exactly the vertices of $N(Q\cap X_1)\cap \overline{X}_1$. On the side of $Q$ the only unmatched vertices are the vertices of $Q\cap R_n=Q\cap X_n$. Thus, the claimed formula on $\delta(Q)$ follows. 
\end{proof}

We have already introduced shiftings along diagonals of the grid, so now we introduce shifting along the rows. Take any $Q\subseteq V_1$. We say that $Q'$ is obtained from $Q$ by the \emph{left shifting} if it is constructed as follows: for each  integer $i\in\{1\ldots n\}$, all the $r=|X_i\cap Q|$ vertices of $X_i\cap Q$ are replaced by $x_1^i,\ldots,x_r^i$. See Fig.~\ref{fig:left} for an example.
Respectively, $Q'$ is obtained from $Q$ by the \emph{right shifting} if for each  integer $i\in\{1,n\}$, all the $r=|X_i\cap Q|$ vertices of $X_i\cap Q$ are replaced by $x_{\ell(i)-r+1}^i,\ldots,x_{\ell(i)}^i$. A pyramidal set $Q\subseteq V_1$ is  called  \emph{left-pyramidal}  if for any $(x,y)\in Q$ with $x\geq 3$ we also have that 
$(x-2,y)\in Q$. Respectively, $Q$ is  \emph{right-pyramidal} if for any $(x,y)\in Q$ with $x\leq n-2$ we also have that $(x+2,y)\in Q$. 

For a pyramidal set $Q\subseteq V_1$, let $i_1(Q)=0$ if $(1,y)\notin Q$ for all $y\in\{1,\ldots,n\}$ and 
$i_1(Q)=\max\{y\, |\, (1,y)\in Q\}$ otherwise. Similarly, let $i_2(Q)=0$ if $(n,y)\notin Q$ for all $y\in\{1,\ldots,n\}$ and 
$i_2(Q)=\max\{y\, |\, (n,y)\in Q\}$ otherwise.

\begin{figure}[ht]
\centering\scalebox{0.75}{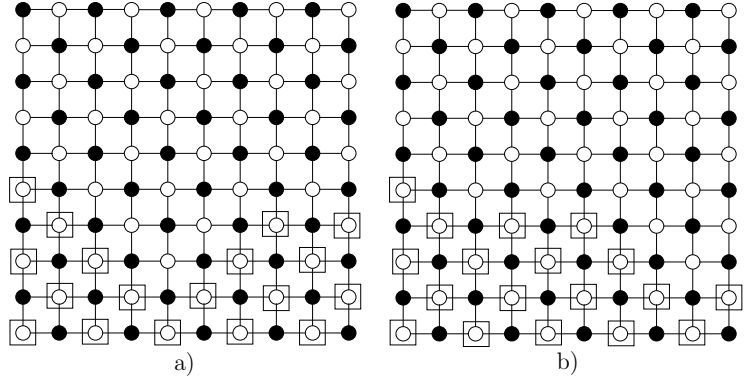}
\caption{An example of the left shifting, $i_1=5$ and $i_2=4$; the vertices of $Q$ (in a) and $Q'$ (in b) are marked by square frames.
\label{fig:left}}
\end{figure}

\begin{lemma}\label{lem:left-right}
Let $Q\subseteq V_1$ be a pyramidal set. 
If $i_1(Q)\geq i_2(Q)$ (as in Fig.~\ref{fig:left}),  then  $Q'$ obtained from $Q$ by the left shifting is left-pyramidal and satisfies $\delta(Q')\leq \delta(Q)$. Respectively,
if $i_1(Q)\leq i_2(Q)$, then $Q'$ obtained from $Q$ by the right shifting is right-pyramidal and satisfies $\delta(Q')\leq \delta(Q)$.
\end{lemma}

\begin{proof}
\begin{comment}
For $i\in\{1,\ldots,n\}$, let $\overline{X}_i=\{(x,y)|1\leq x\leq n, y=i\}\setminus X_i$. Clearly, for any $v\in X_i$, $N_G(v)\in \overline{X}_{i-1}\cup \overline{X}_{i}\cup \overline{X}_{i+1}$ assuming that $\overline{X}_{0}=\overline{X}_{n+1}=\emptyset$.
Notice that for $i\in\{1,\ldots,n-1\}$, $|N_G(Q\cap X_i)\cap \overline{X}_{i+1}|=|Q\cap X_i|$ and this holds for any $Q\subseteq V_1$. Because $Q$ is pyramidal, for any $i\in\{2,\ldots,n\}$, $N_G(Q\cap X_i)\cap (\overline{X}_{i-1}\cup\overline{X}_{i})\subseteq N_G(Q\cap X_{i-1})\cap (\overline{X}_{i-1}\cup\overline{X}_{i})$. It follows that to prove the lemma, it is sufficient to show that
\begin{itemize}
\item[i)] $|N_G(Q\cap X_1)\cap \overline{X}_1|\geq |N_G(Q'\cap X_1)\cap \overline{X}_1|$, and
\item[ii)] $Q'$ is left (or, respectively, right)-pyramidal.
\end{itemize}
\end{comment}

We first prove that if $Q'$ is obtained by the left shifting, then $Q'$ is left-pyramidal, and if $Q'$ is obtained by the right shifting, then $Q'$ is right-pyramidal.
It is straightforward to see that if $Q'$ is obtained by the left shifting, then $(x,y)\in Q'$ implies $(x-2,y)\in Q'$ whenever $x\geq 3$. Symmetrically, 
if $Q'$ is obtained by the right shifting, then $(x,y)\in Q'$ implies $(x+2,y)\in Q'$ whenever $x\leq n-2$. Hence, we only have to prove that $Q'$ is pyramidal, i.e., 
for any $(x,y)\in Q'$ such that $y\geq 2$, it holds that $(x-1,y-1)\in Q'$ provided $x\geq 2$ and $(x+1,y-1)\in Q'$ provided $x\leq n-1$.
Let us fix some $(x,y)\in Q'$ with $y\geq 2$. 

Assume first that $Q'$ is obtained from $Q$ by the left-shifting.
If $X_y\cap Q=X_y$, i.e., $Q$ occupies the whole $X_y$, then because $Q$ is pyramidal, we also have $X_{y-1}\cap Q=X_{y-1}$. From the construction of $Q'$ it follows that $X_y\cap Q'=X_y$ and $X_{y-1}\cap Q'=X_{y-1}$, and the claimed condition holds for $(x,y)$ trivially. 
Suppose then that $X_y\setminus Q\neq \emptyset$.
In such a situation it can be easily seen that $|Q\cap X_y|\geq |Q\cap X_{y-1}|$ because $Q$ is pyramidal, and hence also $|Q'\cap X_y|\geq |Q'\cap X_{y-1}|$. By the construction of $Q'$, we infer that $(x-1,y-1)\in Q'$ provided $x\geq 2$. The second property, i.e. $(x+1,y-1)\in Q'$ provided $x\leq n-1$, also follows if at least one of the following conditions holds:
$|Q\cap X_y|> |Q\cap X_{y-1}|$ or $y$ is odd. Thus, the only remaining case is when $|Q\cap X_y|=|Q\cap X_{y-1}|$ and $y$ is even. Since $X_y\setminus Q\neq \emptyset$ and $Q$ is pyramidal, one can easily verify that the only situation when $|Q\cap X_y|=|Q\cap X_{y-1}|$ is the following: $n$ is even, $Q\cap X_y=\{x_r^y,\ldots,x_{\ell(y)}^y\}$ for some $r\in\{2,\ldots,\ell(y)\}$, where $x_{\ell(y)}^y=(n,y)$, and $Q\cap X_{y-1}=\{x_{r}^{y-1},\ldots,x_{\ell(y-1)}^{y-1}\}$. But then $(n,y)\in Q$ and $(1,y-1)\notin Q$ and we have that $i_1(Q)<y\leq i_2(Q)$. This is a contradiction with the assumption that $Q'$ is obtained by the left shifting.

The arguments for the case when $Q$ is obtained by the right shifting are exactly symmetric, and hence we omit the second check.

\begin{comment}
Assume now that $Q'$ is obtained from $Q$ by the right-shifting. Notice that if $n$ is odd, then, by symmetry, we can use the same arguments as above. Let $n$ be even. 
If $X_y\subseteq Q$, then $X_y\subseteq Q'$ and because $Q$ is pyramidal, $X_{y-1}\subseteq Q'$. Then, $(x-1,y-1)\in Q'$ if $x\geq 2$ and $(x+1,y-1)\in Q'$ if $x\leq n-1$.
Suppose that $X_y\setminus Q\neq \emptyset$.
Because $Q$ is pyramidal, $|Q'\cap X_y|=|Q\cap X_y|\geq |Q\cap X_{y-1}|=|Q'\cap X_{y-1}|$. Therefore, $(x+1,y-1)\in Q'$ if $x\leq n-1$. Also if 
 $|Q'\cap X_y|> |Q'\cap X_{y-1}|$, then $(x-1,y-1)\in Q'$ if $x\geq 2$. Suppose that  $|Q'\cap X_y|=|Q'\cap X_{y-1}|$. 
If $y$ is even, then $(n,y)\in Q'$ and $(x-1,y-1)\in Q'$ if $x\geq 1$. 
Let $y$ be odd. Then because $Q$ is pyramidal and 
$\ell(i)>|Q\cap X_y|=|Q\cap X_{y-1}|$, we obtain that $Q\cap X_y=\{x_1^y,\ldots,x_{r}^y\}$ where $x_{1}^y=(1,y)$
 and $Q\cap X_{y-1}=\{x_{1}^{y-1},\ldots,x_{r}^{y-1}\}$ for $r\in\{1,\ldots,\ell{i}-1\}$. But then $(n-1,y)\notin Q$ and, therefore 
$(n,j)\notin Q$ for $j\geq y-1$. We have that $i_1>i_2$ contradicting the assumption that $Q'$ is obtained by the right shifting.
\end{comment}

We are left with proving that $\delta(Q')\leq \delta(Q)$. Since we already know that both $Q'$ and $Q$ are pyramidal, from Lemma~\ref{lem:pyr-formula} we infer that it suffices to prove that $|N(Q\cap X_1)\cap \overline{X}_1|\geq |N(Q'\cap X_1)\cap \overline{X}_1|$. If $X_1\subseteq Q$, then $Q'\cap X_1=Q\cap X_1=X_1$ and the condition holds trivially. Otherwise, if $X_1\setminus Q\neq \emptyset$, it can be easily seen that $|N(Q\cap X_1)\cap \overline{X}_1|\geq |Q\cap X_1|$. On the other hand, by the construction of $Q'$ we have that $|N(Q'\cap X_1)\cap \overline{X}_1|= |Q'\cap X_1|=|Q\cap X_1|$ apart from the situation when $n$ is even and $Q'$ was obtained by the right shifting; In this case we have $|N(Q'\cap X_1)\cap \overline{X}_1|=|Q\cap X_1|+1$, and this is the only situation left. Observe, however, that provided $n$ is even and $X_1\setminus Q\neq \emptyset$, the only situation with $|N(Q\cap X_1)\cap \overline{X}_1|=|Q\cap X_1|$ is when $Q=\{x_1^1,x_2^1,\ldots,x_r^1\}$ for some $r<\ell(1)$, and otherwise we are done. But then we would have that $(1,1)\in Q$ and $(n-1,1)\notin Q$, which, by $Q$ being pyramidal, implies that $i_1(Q)>0=i_2(Q)$. This is a contradiction with the fact that $Q'$ was obtained by the right shifting.
\end{proof}

\begin{comment}
If $X_1\subseteq Q$, then $Q'\cap X_1=Q\cap X_1$ and i) holds. Assume that $X_1\setminus Q\neq \emptyset$.
Observe that $|N_G(Q\cap X_1)\cap \overline{X}_1|\geq |Q\cap X_1|$ for any $Q$. If $n$ is odd, then $|N_G(Q'\cap X_1)\cap \overline{X}_1|= |Q'\cap X_1|=|Q\cap X_1|\leq |N_G(Q\cap X_1)\cap \overline{X}_1|$.
The same holds if $n$ is even and $Q'$ is obtained from $Q$ by the left shifting. 
It remains to consider the case when $n$ is even and $Q'$ is obtained by the right shifting. In this case $|N_G(Q'\cap X_1)\cap \overline{X}_1|=|Q'\cap X_1|+1$.
Because $X_1\setminus Q\neq \emptyset$, then there is $j\in\{1,\ldots,\ell(1)\}$ such that $x_j^1\notin Q$ and it immediately implies that 
$|N_G(Q\cap X_1)\cap \overline{X}_1|\geq |Q\cap X_1|+1=|Q'\cap X_1|+1=|N_G(Q'\cap X_1)\cap \overline{X}_1|$.
\end{comment}

We say that a left-pyramidal set $Q\subseteq V_1$ has a {\em{left spot at $(i,j)$}} if $(i,j)\in Q$ and $(i-1,j+1)\in V(G)\setminus Q$. Similarly, a right-pyramidal set $Q\subseteq V_1$ has a {\em{right spot at $(i,j)$}} if $(i,j)\in Q$ and $(i+1,j+1)\in V(G)\setminus Q$. Obviously, all the left spots of a left pyramidal set have pairwise different $x$-coordinates, and the same also holds for right spots of right pyramidal sets. 

The following two lemmas can be easily verified by a direct check using Lemma~\ref{lem:pyr-formula}.

\begin{lemma}\label{lem:terr-left}
Let $Q\subseteq V_1$ be a left-pyramidal set with two different left spots $(i_1,j_1)$ and $(i_2,j_2)$, such that $(i_1,j_1)$ and $(i_2,j_2)$ have the smallest and the largest $x$-coordinates among the left spots of $Q$, respectively. Construct $Q'=Q\setminus \{(i_2,j_2)\}\cup \{(i_1-1,j_1+1)\}$. Then $Q'$ is also a left-pyramidal set and $\delta(Q')\leq \delta(Q)$.
\end{lemma}

\begin{lemma}\label{lem:terr-right}
Let $Q\subseteq V_1$ be a right-pyramidal set with two different right spots $(i_1,j_1)$ and $(i_2,j_2)$, such that $(i_1,j_1)$ and $(i_2,j_2)$ have the largest and the smallest $x$-coordinates among the right spots of $Q$, respectively. Construct $Q'=Q\setminus \{(i_2,j_2)\}\cup \{(i_1+1,j_1+1)\}$. Then $Q'$ is also a right-pyramidal set and $\delta(Q')\leq \delta(Q)$.
\end{lemma}

\begin{figure}[ht]
\centering\scalebox{0.75}{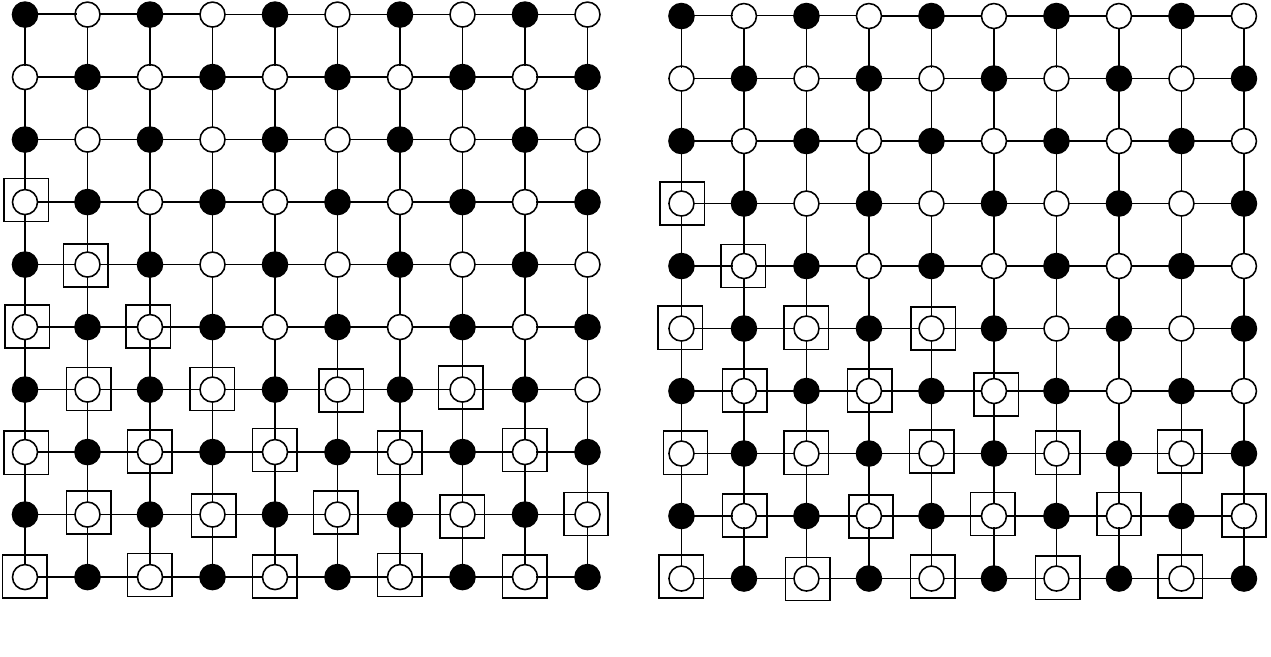}
\caption{Replacement of Lemma~\ref{lem:terr-left} for $(i_1,j_1)=(4,6)$ and $(i_2,j_2)=(4,8)$; the vertices of $Q$ (in a) and $Q'$ (in b) are marked by square frames.
\label{fig:terr-1}}
\end{figure}

\begin{figure}[ht]
\centering\scalebox{0.75}{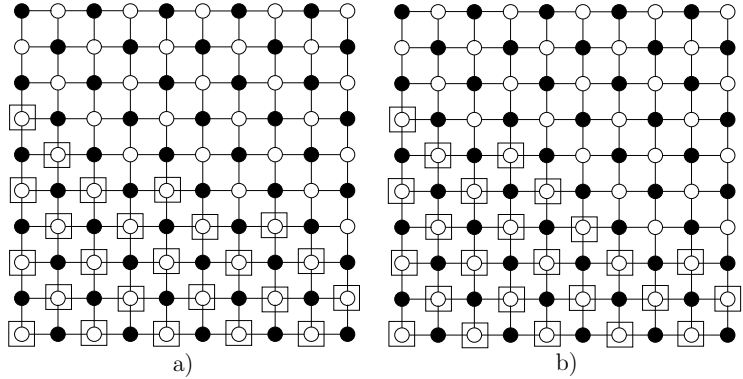}
\caption{Replacement of Lemma~\ref{lem:terr-left} for $(i_1,j_1)=(5,5)$ and $(i_2,j_2)=(4,8)$; the vertices of $Q$ (in a) and $Q'$ (in b) are marked by square frames.
\label{fig:terr-2}}
\end{figure}

Note that the transformations of Lemmas~\ref{lem:terr-left} and~\ref{lem:terr-right} can be applied as long as the set $Q$ in question has at least two left (resp. right) spots.

Recall that $V_1=U_2\cup U_4\cup\ldots\cup U_{2n}$. We define the ordering $v_1,\ldots,v_{\lceil n^2/2\rceil}$ of the vertices of $V_1$ as follows: the sequence enumerates consequently the vertices of $U_2,U_4,\ldots,U_{2n}$ and the vertices of each $U_i$ are listed in the order $u_1^i,\ldots,u_{s(i)}^i$. For a positive integer $p$, let $Z_p=\{v_1,\ldots,v_p\}$. Recall also that 
$V_1=W_{2-2\lceil n/2\rceil}\cup\ldots\cup W_{-2}\cup W_{0}\cup W_2\cup\ldots\cup W_{2\lceil n/2\rceil-2}$. Respectively, we define another ordering $v_1',\ldots,v_{\lceil n^2/2\rceil}'$ of the vertices of $V_1$: the sequence enuemrates consequently the vertices of $W_{2\lceil n/2\rceil-2},W_{2\lceil n/2\rceil-4},\ldots,W_{2-2\lceil n/2\rceil}$ and the vertices of each $W_i$ are listed in the order $w_1^i,\ldots,w_{t(i)}^i$.
For a positive integer $p$, we define $Z_p'=\{v_1',\ldots,v_p'\}$.

\begin{theorem}\label{thm:iso}
Let $G$ be an $(n\times n)$-grid. Then for any $Q\subseteq V_1$, it holds that $\delta(Q)\geq \min(\delta(Z_p),\delta(Z_p'))$, where $p=|Q|$.
\end{theorem}

\begin{proof}
Let $p\in \{1,2,\ldots,\lceil \frac{n^2}{2}\rceil\}$ and let $\delta^*=\min\{\delta(Q)\, |\, Q\subseteq V_1,|Q|=p\}$. 

First, we show that there is a left-pyramidal or a right-pyramidal set $Q\subseteq V_1$ of size $p$ with $\delta(Q)=\delta^*$. Let $Q$ be any set of size $p$ with $\delta(Q)=\delta^*$. Assume that $Q$ is chosen in such a way that the sum of $y$-coordinates of the vertices of $Q$ is minimum. Then $Q$ is pyramidal because otherwise we could apply the down-right or down-left shifting and obtain a set $Q'$ that would have smaller sum of $y$-coordinates of its vertices, and for which it would hold that $\delta(Q')\leq \delta(Q)$ by Lemma~\ref{lem:down}. 
Suppose now that $Q$ is neither left- nor right-pyramidal. If $i_1(Q)\geq i_2(Q)$, then  let $Q'$ be the set obtained from $Q$ by the left shifting. By Lemma~\ref{lem:left-right}, we have that $\delta(Q')\leq \delta(Q)$ and $Q'$ is left-pyramidal. On the other hand, if $i_1(Q)<i_2(Q)$, then the set $Q'$  obtained from $Q$ by the right shifting is right-pyramidal and satisfies $\delta(Q')\leq \delta(Q)$ by Lemma~\ref{lem:left-right}. Since $\delta(Q)=\delta^*$ is minimum possible, in both cases we conclude that $\delta(Q')=\delta^*$.

Suppose now that there is a left pyramidal set $Q$ of size $p$ with $\delta(Q)=\delta^*$. Among all such sets we select $Q$ for which the sum of $x$-coordinates of its vertices is minimum. Then $Q$ has at most one spot, since otherwise using Lemma~\ref{lem:terr-left} we could construct a left-pyramidal set $Q'$ with $\delta(Q')\leq \delta(Q)$ and a smaller sum of $x$-coordinates of vertices. It remains to notice that if a left-pyramidal set of size $p$ has at most one spot, then in fact $Q=Z_p$.

The case when there is a right-pyramidal set $Q$ of size $p$ with $\delta(Q)=\delta^*$ is symmetric. Among all such sets we select $Q$ that maximizes the sum of $x$-coordinates of its vertices, and using Lemma~\ref{lem:terr-right} we argue that then $Q=Z_p'$. 

Summarizing, in the first case we have that $\delta^*=\delta(Z_p)$ and in the second we have that $\delta^*=\delta(Z_p')$, so we infer that $\delta^*=\min(\delta(Z_p),\delta(Z_p'))$.
\end{proof}

Using Theorem~\ref{thm:iso}, it is possible to obtain an explicit expression for the tight lower bound for $\delta(Q)$, but such an expression is rather ugly. In particular, it can be noticed that there are cases when the bound is given by $\delta(Z_p)$ and cases when $\delta(Z_p')$ is minimum. Consider, e.g., the $(6\times 6)$-grid.  Then $6=\delta(Z_4)<\delta(Z_4')=7$ and $15=\delta(Z_{12})>\delta(Z_{12}')=14$. To compute the hunter number of a grid, we need the bound for one special case.

\begin{corollary}\label{cor:bound}
Let $G$ be an $(n\times n)$-grid, where $n\geq 4$ and $n$ is even. Then for any $Q\subseteq V_1$ with $|Q|=\frac{n^2}{4}-\frac{n}{2}$, it holds that $\delta(Q)\geq \frac{n^2}{4}$.
\end{corollary}

\begin{proof}
Let $p=\frac{n^2}{4}-\frac{n}{2}$.
Observe that $Z_p$ contains all the vertices of $U_2,\ldots,U_{n-2}$ and $\frac{n}{2}-1$ vertices of $U_n$. 
Then $N(Z_p)$ contains all the vertices of $U_3,\ldots,U_{n-1}$ and $\frac{n}{2}$ vertices of $U_{n+1}$. 
Because $s(3)+\ldots+s(n-1)=2+\ldots+(n-2)=\frac{n^2}{4}-\frac{n}{2}$, we have that $\delta(Z_p)= \frac{n^2}{4}$.
For the set $Z_p'$, we have that $Z_p'=W_{n-2}\cup\ldots \cup W_2$ and $N(Z_p')=W_{n-1}\cup\ldots \cup W_1$. Hence, it follows that $\delta(Z_p')= \frac{n^2}{4}$.
By Theorem~\ref{thm:iso}, $\delta(Q)\geq \min(\delta(Z_p),\delta(Z_p'))=\frac{n^2}{4}$.
\end{proof}

\subsection{The hunter number of a grid}\label{sec:hunt-grid}
Now we are ready to compute the hunter number of a grid.

\begin{theorem}\label{thm:hunt-grid}
Let $G$ be an $(n\times m)$-grid. Then $h(G)=\lfloor \frac{\min\{n,m\}}{2}\rfloor+1$.
\end{theorem}

\begin{proof}
Recall that we assume that $n\leq m$.

 First, we prove that $h(G)\leq \lfloor \frac{n}{2}\rfloor+1$. By Proposition~\ref{prop:subgr}, it is sufficient to show it for odd $n$. Therefore, we assume that $n$ is odd and, using Lemma~\ref{lem:bip},  construct a winning strategy for $\frac{n-1}{2}+1$ hunters with respect to $V_1$.  Consecutively, for  $i=1,\ldots,m-1$, the hunter player makes the following sequence of shoots as it is shown in Fig.~\ref{fig:shoots}:
\begin{itemize}
\item shoot at $(i,1), (i,3),\ldots,(i,n)$,
\item for $j=1,\ldots,(n+1)/2$, shoot at $(i+1,1),(i+1,3)\ldots,(i+1,2j-1),(i,2j),\ldots,(i,n-1)$ and then
at $(i+1,2),\ldots,(i+1,2j),(i,2j+1),\ldots,(i,n)$. 
\end{itemize}
Finally, the hunter player shoots at $(m,1), (m,3),\ldots,(m,n)$.
It is straightforward to verify (see Fig.~\ref{fig:shoots}) that the following claim holds for each $i\in\{1,\ldots,m\}$:
after the $i$-th series of rounds,
\begin{itemize}
\item the vertices of $V_{2-i\bmod 2}$ are clear;
\item the vertices $(x,y)$ for $x>i$ are clear.
\end{itemize}
This immediately implies that we have a winning hunters' strategy.

\begin{figure}[ht]
\centering\scalebox{0.75}{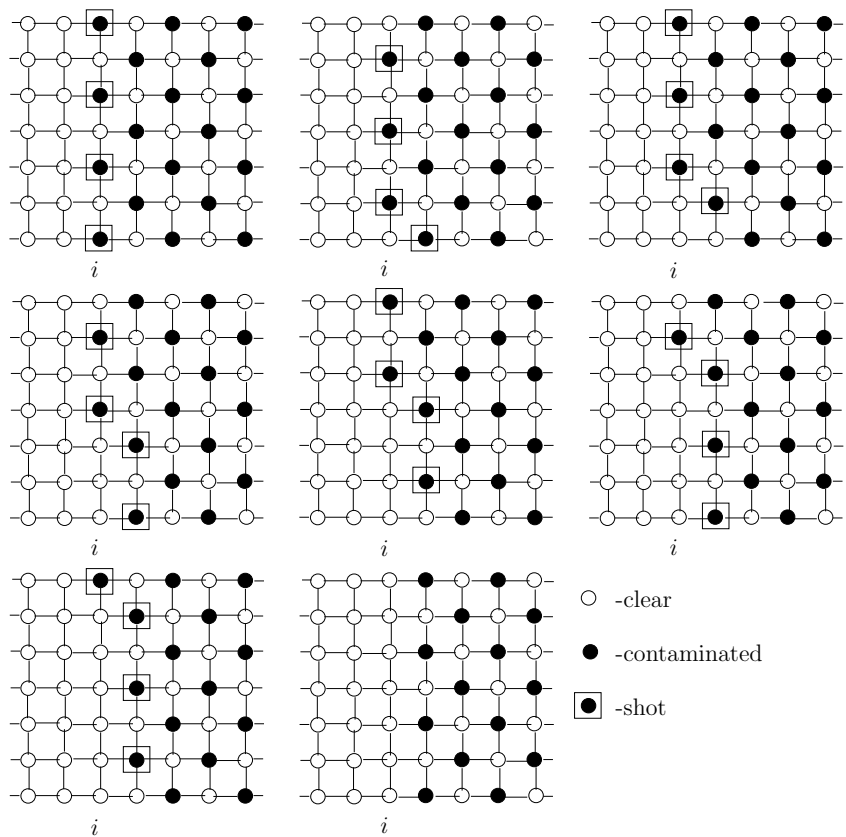}
\caption{The series of shots for each $i\in\{1,\ldots,m-1\}$ for $n=7$.
\label{fig:shoots}}
\end{figure}

Now we show that $h(G)\geq \lfloor \frac{n}{2}\rfloor+1$. By Proposition~\ref{prop:subgr}, it is sufficient to show it for $(n\times n)$-grids for even $n$. If $n=2$, then a direct check shows that $h(G)=2$ and the claim holds. Suppose then that $n\geq 4$. We show that the Hunter player has no winning strategy for $\frac{n}{2}$ hunters. 

For the sake of contradiction, suppose $\mathcal{H}=(H_1,H_2,\ldots )$ is Hunter's strategy for $\frac{n}{2}$ hunters. We show inductively that for every $i\geq 1$, each of the sets $V_1$ and $V_2$ has at least $\frac{n^2}{4}$ contaminated vertices after the $i$-th shot. Clearly, all vertices are contaminated before the first shoot. Assume that the claim holds before $i$-th round. Set $V_1$ contains at least $\frac{n^2}{4}$ contaminated vertices before the $i$-th shot; let us denote this set of contaminated vertices by $A_{i-1}$. As $|H_i|\leq \frac{n}{2}$, we have a set $Q=A_{i-1}\setminus H_i$ of at least $\frac{n^2}{4}-\frac{n}{2}$ vertices that were contaminated before the $i$-th shot and were not shot during the $i$-th round. By applying Corollary~\ref{cor:bound} to any subset of $Q$ of size exactly $\frac{n^2}{4}-\frac{n}{2}$, we infer that $\delta(Q)\geq \frac{n^2}{4}$ and hence at least $\frac{n^2}{4}$ vertices of $V_2$ are contaminated after the $i$-th shot. To show the symmetric claim for $V_2$, we can use exactly the same arguments, because we can apply Corollary~\ref{cor:bound} also to subsets of $V_2$; this follows from the assumption that $n$ is even and, therefore, $V_1$ and $V_2$ can be mapped to each other by an automorphism of $G$. 

This is a contradiction with the assumption that $\mathcal{H}$ is a winning strategy for the Hunter player. As $\mathcal{H}$ was chosen arbitrarily, we have that $\frac{n}{2}$ hunters cannot hunt the rabbit.
\end{proof}

\section{Hunting rabbit on a tree}\label{sec:tree}
In this section we provide upper and lower bounds on the hunting number of a tree. 

\subsection{Upper bound}
For the  upper bound we show that   the hunting number of a graph does not exceed its pathwidth plus one. Then 
the bound will follow from the well-known bound on the pathwidth of a tree. 

A \emph{path decomposition} of a graph $G$ is a sequence $(X_1,\ldots,X_{\ell})$ of subsets
of $V(G)$ (called {\em bags})
such that 
\begin{itemize}
\item[i)] $\bigcup_{i \in V(T)} X_{i} = V(G)$,
\item[ii)] for each edge $xy \in E(G)$, $x,y\in X_i$ for some  $i\in V(T)$, and
\item[iii)] for each $x\in V(G)$, if $x\in X_i\cap X_j$ for some $1\leq i\leq j\leq \ell$, then 
$x\in X_k$ for all $k$ with $i\leq k\leq j$. 
\end{itemize}
The \emph{width} of a path decomposition $(X_1,\ldots,X_{\ell})$ is 
$\max\{|X_i|\mid 1\leq i\leq \ell\}-1$. The \emph{pathwidth} of a graph $G$ (denoted as $\pw(G)$) is the minimum width over all path decompositions of $G$.

\begin{proposition}\label{prop:pw}
For a graph $G$ it holds that $h(G)\leq \pw(G)+1$, and this bound is tight for graphs of pathwidth at least 2.
\end{proposition}

\begin{proof}
Let $(X_1,\ldots,X_{\ell})$ be a path decomposition of $G$ of width $k=\pw(G)$. 
We show that $\mathcal{H}=(X_1,\ldots,X_{\ell})$ is a winning hunters' strategy for  $k+1$ hunters.
To prove this, we show that all the vertices of $(\cup_{j=1}^iX_j)\setminus X_{i+1}$ are clear after the $i$-th round, for all $i\in\{1,\ldots,\ell\}$; we assume here that $X_{\ell+1}=\emptyset$.  
It is straightforward to see that the claim holds for $i=1$, because a vertex $v\in X_1$ can have a neighbor outside $X_1$ only if $v\in X_1\cap X_2$, by the conditions ii) and iii) of the definition of a path decomposition. Assume that the claim holds for some $i\in\{1,\ldots,\ell\}$. If a vertex $v\in\cup_{j=1}^{i+1}X_j$ is contaminated after the $i+1$-st round, then $v$ has to be adjacent to a vertex $u$ that was contaminated after the $i$-th round, and moreover $u\notin X_{i+1}$. By the inductive assumption we infer that $u\notin\cup_{j=1}^{i+1}X_j$, and hence by ii) and iii) of the definition of a path decomposition it follows that $v\in X_{i+1}\cap X_{i+2}$. Thus, no vertex of $(\cup_{j=1}^{i+1}X_j)\setminus X_{i+2}$ is contaminated after the $i+1$-st round, which proves the induction step. It remains to observe that after the $\ell$-th round all the vertices of $G$ are clear. This means that the strategy $\mathcal{H}$ is winning.

\begin{figure}[ht]
\centering\scalebox{0.75}{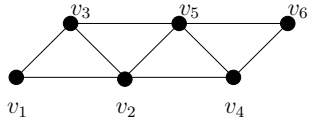}
\caption{An example of a graph $G$ with $h(G)=\pw(G)+1=3$. 
\label{fig:pw}}
\end{figure}

Now we show tightness of the bound. First, we prove that the bound is tight for graphs of pathwidth 2.

Consider the graph $G$ shown in Fig.~\ref{fig:pw}. It is straightforward to verify that $\pw(G)=2$. We show that $h(G)\geq 3$.  
Consider an arbitrary hunters' strategy $\mathcal {H}=(H_1,H_2,\ldots)$ for 2 hunters. 
We prove that $\mathcal{H}$ cannot be a winning strategy by showing that for any $i\geq 1$, after the $i$-th round the following invariant holds: at least $5$ vertices of $G$ are contaminated, and the only vertices that can be clear are $v_1,v_3,v_4,v_6$. We shall denote this invariant by $(\diamondsuit)$.

Clearly, all the vertices are contaminated in the beginning, so $(\diamondsuit)$ holds before round $1$. Suppose now that $(\diamondsuit)$ is satisfied before the $i$-th round and we show that the same holds after the round. By symmetry and monotonicity under containment, it is sufficient to consider two cases.

\medskip
\noindent
{\bf Case 1.} Vertex $v_1$ is clear and all other vertices are contaminated before the $i$-th round. Notice that each of $v_2,v_4,v_5$ has at least 3 contaminated neighbors before the $i$-th round. Hence, they are contaminated after the round as well. If $H_i=\{v_4,v_5\}$, then $v_1,\ldots,v_5$ are contaminated after the round and we have $(\diamondsuit)$. Assume then that $H_i\neq\{v_4,v_5\}$. Then $v_6$ is contaminated after the round. If $H_i=\{v_2,v_5\}$, then 
$v_1$ is contaminated after the round and we have $(\diamondsuit)$, because $v_1,v_2,v_4,v_5,v_6$ are contaminated after the round. If 
$H_i\neq\{v_2,v_5\}$, then 
$v_3$ gets contaminated and we have $(\diamondsuit)$, because $v_2,v_3,v_4,v_5,v_6$ are contaminated after the round.
 
\medskip
\noindent
{\bf Case 2.} Vertex $v_3$ is clear and all other vertices are contaminated before the $i$-th round. Notice that each of $v_2,v_3,v_4,v_5$ has at least 3 contaminated neighbors before the $i$-th round. Hence, they are contaminated after the round.  
If $H_i=\{v_4,v_5\}$, then $v_1,\ldots,v_5$ get contaminated and we have $(\diamondsuit)$ after the round. 
If $H_i\neq\{v_4,v_5\}$, then $v_2,\ldots,v_6$ get contaminated
and we again have $(\diamondsuit)$ after the round.

\medskip
To show tightness of the bound for graphs of pathwidth $k\geq 2$, consider the graph $G'$ obtained from the graph $G$ shown in Fig.~\ref{fig:pw} as follows. 
We add a set $X$ of $k-2$ vertices and join them pairwise by edges to form a clique. Then every vertex of $X$ is joined by an edge with every vertex of $G$. 
It is straightforward to see that $\pw(G')=k$. 
We show that $h(G')=k+1$. 
Let $\mathcal {H}=(H_1,H_2,\ldots)$ be an arbitrary hunters' strategy for $k$ hunters. We prove that $\mathcal{H}$ is not a winning strategy by showing that for any $i\geq 1$, after the $i$-th round the following invariant $(\diamondsuit\diamondsuit)$ holds:  the invariant $(\diamondsuit)$ if fulfilled for the vertices of $G$ and the vertices of $X$ are contaminated.

As all the vertices are contaminated in the beginning, $(\diamondsuit\diamondsuit)$  holds before round $1$. Suppose that $(\diamondsuit\diamondsuit)$ is satisfied before the $i$-th round and we show that the same holds after the round. 

If $X\subseteq H_i$, then at most 2 hunters shoot at the vertices of $G$ and, therefore, $(\diamondsuit)$ holds for $G$ as it was shown above.  Also in this case all the verties of $X$ are contaminated after the $i$-th round, because there is at least one contaminated before $i$-the round vertex $u$ of $G$ such that $u\notin H_i$.  We conclude that $(\diamondsuit\diamondsuit)$ is fulfilled.  

Suppose that $|X\setminus H_i|=1$. Since the vertices of $X\setminus H_i$ are contaminated before the $i$-th round, all the vertices of $G$ are contaminated after $i$-th round. Since at most 3 hunters shoot at the vertices of $G$ and $G$ has at least 5 contaminated vertices before the $i$-th round, the vertices of $X$ are contaminated after the round. Hence, $(\diamondsuit\diamondsuit)$ holds.

Finally, assume that $|X\setminus H_i|\geq 2$ and consider distinct $x,y\in X\setminus H_i$. As $x$ is contaminated before the $i$-round, we have that all the vertices of $G$ are contaminated after $i$-th round. It remains to observe that all the vertices of $X\setminus\{x\}$  are contaminated after the $i$-th round, because $x$ is contaminated before the round, and $x$ is contaminated, because $y$ is contaminated before the $i$-th round. We again have that $(\diamondsuit\diamondsuit)$ holds.
\end{proof}

We proved that the bound is tight for graphs of pathwidth at least 2. It can be noticed that if $\pw(G)=1$, then $h(G)=1$, because every component of $G$ is a caterpillar in this case,  and as it was shown in \cite{BritnellW13}, in this case  $h(G)=1$.

Since the pathwidth of an $n$-vertex tree is bounded by $\Oh(\log n)$
\cite{Parsons78,Petrov82}, we obtain the following theorem.
\begin{theorem} For every $n$-vertex tree $T$,  $\hunt(T)\leq \Oh(\log n)$. 
\end{theorem}

\subsection{Lower bound}
In this section we prove that the hunting number of an $n$-vertex tree can be as large as $\Omega(\log n/\log \log n)$. More precisely, we prove the following theorem.

\begin{theorem}\label{thm:trees-lb}
For every positive integer $k$ there exists a tree $T_k$ such that $|V(T_k)|=2^{\Oh(k\log k)}$ and $\hunt(T_k)\geq k$.
\end{theorem}

The rest of this section is devoted to the proof of Theorem~\ref{thm:trees-lb}. The construction of the sequence of trees $(T_k)_{k=1,2,3,\ldots}$ is inductive. We are going to think of each $T_i$ as of a rooted tree. For $T_1$ we take simply a path on three vertices with the middle vertex being the root. Let us define
$$p(k)=2\cdot ((4k-3)(k-1)+1).$$
To construct $T_k$ based on $T_{k-1}$, perform the following:
\begin{itemize}
\item Create the root $u$.
\item Add $p:=p(k)$ children of $u$, denoted by $v^1,v^2,\ldots,v^{p}$.
\item For every child $v^i$ of $u$, add $p$ subtrees $Q^{i,j}$ for $j=1,2,\ldots,p$, all isomorphic to $T_{k-1}$ and with roots being children of $v^i$.
\end{itemize}
See Figure~\ref{fig:tk} for an illustration. For $i=1,2,\ldots,p$, by $P^i$ we denote the subtree of $T_k$ rooted at $v^i$. Furthermore, let $w^{i,j}$ be the root of subtree $Q^{i,j}$, for all $i,j\in \{1,2\ldots,p\}$. By somehow abusing the notation, we will identify each subtree $P^i$ and $Q^{i,j}$ with its vertex set.

\begin{figure}[t]
                \centering
                \def\svgwidth{\textwidth}
                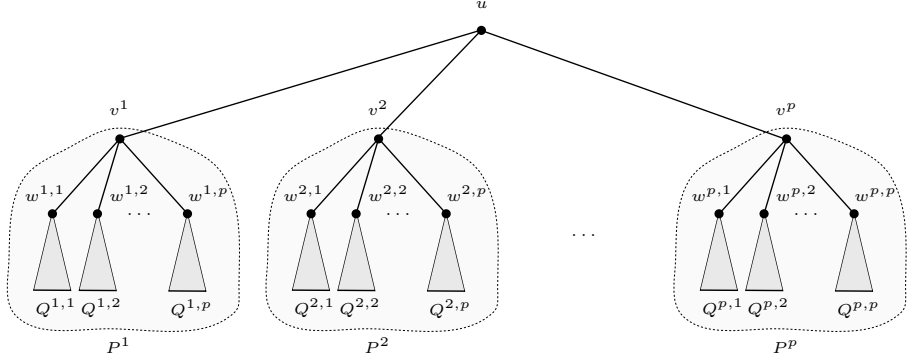
\caption{Construction of tree $T_k$.}\label{fig:tk}
\end{figure}

Observe now that we have recursive equation $|V(T_k)|=1+p(k)+p(k)^2\cdot |V(T_{k-1})|$, from which it immediately follows that $|V(T_k)|=2^{\Oh(k\log k)}$. We are left with proving by induction that $\hunt(T_k)\geq k$ for all positive integers $k$. For $k=1$ we have $\hunt(T_1)=1$, so we proceed to the inductive step for $k\geq 2$. In the following, we denote $T=T_k$.

Let us fix the bipartition $(V_1,V_2)$ of $T$ such that $u\in V_1$ and $\{v^1,v^2,\ldots,v^p\}\subseteq V_2$. We shall prove that $k-1$ hunters do not have a winning strategy on $T$ with respect to $V_1$, which by Lemma~\ref{lem:bip} is equivalent to the main claim. For the sake of contradiction, suppose that there is a winning strategy with respect to $V_1$ for $k-1$ hunters, and denote it by $\mathcal{H}=(H_1,H_2,H_3,\ldots,H_m)$. Therefore, in the beginning of this strategy all the vertices of $V_1$ are contaminated, and at the end all the vertices of $T$ are clean. For $t=0,1,2,\ldots,m$, let $A_t$ be the set of contaminated vertices between hunters' shots $t$ and $t+1$. Thus, $A_0=V_1$, $A_m=\emptyset$, $A_t\subseteq V_1$ for even $t$, and $A_t\subseteq V_2$ for odd $t$.

Let us fix a moment $t\in \{0,1,\ldots,m\}$, and consider subtree $Q^{i,j}$. We shall say that
\begin{itemize}
\item $Q^{i,j}$ is {\em{contaminated at moment $t$}} if $Q^{i,j}\cap A_t\neq \emptyset$.
\item $Q^{i,j}$ is {\em{full at moment $t$}} if $Q^{i,j}\cap A_t=Q^{i,j}\cap V_1$ provided that $t$ is even, and $Q^{i,j}\cap A_t=Q^{i,j}\cap V_2$ provided that $t$ is odd.
\item $Q^{i,j}$ is {\em{well-contaminated at moment $t$}} if the following holds: Supposing $t'\leq t$ is the latest moment not later than $t$ when $Q^{i,j}$ was full, then $|Q^{i,j}\cap H_{t''}|<k-1$ for all $t'<t''\leq t$. In other words, since the last time $Q^{i,j}$ was full, it did not happen that {\em{all}} the available hunters were shooting at $Q^{i,j}$ in some round.
\end{itemize}
Observe that at moment $t=0$ all the subtrees $Q^{i,j}$ are full, hence the definition of being well-contaminated is valid. Observe also that by the induction hypothesis and Lemma~\ref{lem:bip}, each subtree $Q^{i,j}$ cannot be cleaned using less than $k-1$ hunters and beginning from any moment when it is full. This justifies the following claim.

\begin{claim}\label{cl:well-cont-cont}
If a subtree $Q^{i,j}$ is well-contaminated at moment $t$, then it is also contaminated at moment $t$.
\end{claim}

Similarly as for $Q^{i,j}$, a subtree $P^i$ is {\em{contaminated at moment $t$}} if $P^{i}\cap A_t\neq \emptyset$. Also, $P^i$ is {\em{full at moment $t$}} if $P^i\cap A_t=P^i\cap V_1$ provided $t$ is even, and $P^i\cap A_t=P^i\cap V_2$ provided $t$ is odd.

We now prove a few auxiliary observations that will be used in the main proof. 

\begin{claim}\label{cl:getting-full}
The following holds:
\begin{itemize}
\item Suppose that subtree $Q^{i,j}$ is contaminated at moment $t$ and for every $t'$ with $t<t'\leq t+(4k-6)$ we have that $H_{t'}\cap Q^{i,j}=\emptyset$. Then $Q^{i,j}$ is full at moment $t+(4k-6)$.
\item Suppose that subtree $P^{i}$ is contaminated at moment $t$ and for every $t'$ with $t<t'\leq t+(4k-4)$ we have that $H_{t'}\cap Q^{i,j}=\emptyset$. Then $P^{i}$ is full at moment $t+(4k-4)$.
\end{itemize}
\end{claim}
\begin{proof}
The claim follows immediately from the facts that the diameter of each $Q^{i,j}$ is equal to $4k-6$ and the diameter of each $P^{i}$ is equal to $4k-4$. This, in turn, follows from the observation that the diameter of $T_k$ is equal to $4k-2$, which can be proved via a straightforward induction.
\cqed\end{proof}

\newcommand{\cnt}[2]{n_{#1}(#2)}

For a moment $t$ ($0\leq t\leq m$) and index $i$ ($1\leq i\leq p$), let $\cnt{i}{t}$ be the number of subtrees $Q^{i,j}$, for $j=1,2,\ldots,p$, that are well-contaminated at moment $t$. We shall say that
\begin{itemize}
\item $P^i$ is {\em{lightly contaminated at moment $t$}} if $\cnt{i}{t}\leq p/2$;
\item $P^i$ is {\em{heavily contaminated at moment $t$}} if $p/2<\cnt{i}{t}$.
\end{itemize}

\begin{claim}\label{cl:root-contam}
Suppose $P^i$ is heavily contaminated at moment $t$, where $t$ is odd. Then $v^i\in A_t$.
\end{claim}
\begin{proof}
Suppose first that there was a moment $t_0$ with $t-(4k-5)\leq t_0<t$, such that $v^i\in A_{t_0}$ but $v^i\notin H_{t_0+1}$; in particular $t_0$ is odd. Then it follows that $w^{i,j}\in A_{t_0+1}$ for every $j=1,2,\ldots,p$, so in particular all the subtrees $Q^{i,j}$ became contaminated at moment $t_0+1$. Out of these subtrees, at most $(4k-5)(k-1)$ might contain some vertex of $H_{t'}$ for any $t'$ with $t_0<t'\leq t$, which leaves at least one subtree $Q^{i,j}$ that did not contain any shots during all these moments. Since both $t_0$ and $t$ are odd and $Q^{i,j}$ consists of more than one vertex, we infer that $w^{i,j}$ remained contaminated at all the even moments between $t_0+1$ and $t-1$, so in particular $w^{i,j}\in A_{t-1}$. Since $w^{i,j}\notin H_t$, we have that $v^i\in A_t$.

Suppose now that no such moment $t_0$ exists. This means that all the subtrees $Q^{i,j}$ that are contaminated at moment $t$, needed to be contaminated at all moments $t'$ with $\max(t-(4k-5),0)\leq t'\leq t$: the only way a subtree $Q^{i,j}$ can become contaminated is by not shooting at $v^i$ when it is contaminated. At most $(4k-5)(k-1)$ of these subtrees might contain some vertex of $H_{t'}$ for any $t'$ with $\max(0,t-(4k-5))<t'\leq t$, which leaves at least one subtree $Q^{i,j}$ that did not contain any shots during all these moments. By Claim~\ref{cl:getting-full}, this subtree is full at moment $t-1$. Since $t$ is odd, this means that $w^{i,j}\in A_{t-1}$. As $w^{i,j}\notin H_t$, we again have that $v^i\in A_t$.
\cqed\end{proof}

\begin{claim}\label{cl:becomes-well-contam}
Suppose $t$ is an odd moment, $v^i\in A_t$, and $v^i\notin H_{t+1}$. Then $P^i$ is heavily contaminated at moment $t+(4k-5)$.
\end{claim}
\begin{proof}
By the assumption we have that $w^{i,j}\in A_{t+1}$ for every $j=1,2,\ldots,p$. At most $(4k-6)(k-1)$ subtrees $Q^{i,j}$ can contain some vertex of $H_{t'}$ for $t+1<t'\leq t+(4k-5)$. This leaves more than $p/2$ subtrees $Q^{i,j}$ that do not contain any shots during these moments. By Claim~\ref{cl:getting-full}, all these subtrees are full at moment $t+(4k-5)$, so in particular they are well-contaminated then.
\cqed\end{proof}

Finally, we introduce a similar classification for the whole tree $T$ as for subtrees $P^i$. We shall say that 
\begin{itemize}
\item $T$ is {\em{lightly contaminated at moment $t$}} if the number of heavily contaminated subtrees $P^i$ is at most $p/2$;
\item $T$ is {\em{heavily contaminated at moment $t$}} if the number of heavily contaminated subtrees $P^i$ is more than $p/2$.
\end{itemize}
The following two claims can be proved in exactly the same manner as Claims~\ref{cl:root-contam} and~\ref{cl:becomes-well-contam}, with the modification that we use the second point of Claim~\ref{cl:getting-full} instead of the first one.

\begin{claim}\label{cl:whole-root-contam}
Suppose $T$ is heavily contaminated at moment $t$, where $t$ is even. Then $u\in A_t$.
\end{claim}

\begin{claim}\label{cl:whole-becomes-well-contam}
Suppose $t$ is an even moment, $u\in A_t$, and $u\notin H_{t+1}$. Then $T$ is heavily contaminated at moment $t+(4k-3)$.
\end{claim}

We are finally ready to prove that $\hunt(T_k)\geq k$ by exposing that the existence of the hunters' strategy $(H_1,H_2,H_3,\ldots,H_m)$ leads to a contradiction. Let $t\in \{0,1,\ldots,m\}$ be the latest moment when $T$ was heavily contaminated. Since $T$ is heavily contaminated at moment $0$, and lightly contaminated at moment $m$, we infer that such $t$ exists and satisfies $t<m$. By the maximality of $t$, we have that $T$ is lightly contaminated at moment $t+1$. This means that some subtree $P^{i_0}$ ceased to be highly contaminated between moments $t$ and $t+1$, which in turn means that some subtree $Q^{i_0,j_0}$ ceased to be well-contamined between moments $t$ and $t+1$. By the definition of being well-contaminated, subtree $Q^{i_0,j_0}$ could have ceased to be well-contaminated only if all the $k-1$ hunters were shooting at it at moment $t+1$, i.e., we have that $H_{t+1}\subseteq Q^{i_0,j_0}$. This implies in particular that $H_{t+1}\cap \{u,v^1,v^2,\ldots,v^p\}=\emptyset$.

We now consider two cases depending on the parity of $t$.

Suppose first that $t$ is odd, so $A_t\subseteq V_2$. We have that more than $p/2$ subtrees $P^i$ are heavily contaminated at moment $t$, including the subtree $P^{i_0}$. By Claim~\ref{cl:root-contam}, for each of these subtrees $P^i$ we have that $v^i\in A_t$. Since no vertex $v^i$ is contained in $H_{t+1}$, we can apply Claim~\ref{cl:becomes-well-contam} and infer that every subtree $P^i$ that was heavily contaminated at moment $t$ is again heavily contaminated at moment $t+(4k-5)$. Hence, at moment $t+(4k-5)$ we again have more than $p/2$ heavily contaminated subtrees $P^i$, a contradiction with the maximality of $t$.

Suppose now that $t$ is even, so $A_t\subseteq V_1$. Since $T$ is heavily contaminated at moment $t$, by Claim~\ref{cl:whole-root-contam} we have that $u\in A_t$. Since $u\notin H_{t+1}$, by Claim~\ref{cl:whole-becomes-well-contam} we have that $T$ is again heavily contaminated at moment $t+(4k-3)$, which contradicts the maximality of $t$.

We have obtained a contradiction in both of the cases, so this completes the inductive proof that $\hunt(T_k)\geq k$. Thus Theorem~\ref{thm:trees-lb} is proven.

\section{Conclusions}\label{sec:conclusions}
We conclude with a few open questions. 
\begin{itemize}
\item We have shown that 
 the extremal value of the hunting number for $n$-vertex trees is between $\Omega(\log n/\log \log n)$ and $\Oh(\log n)$. Unfortunately, we are unable to close this gap, which constitutes the first natural question.
\item We leave the algorithmic aspects of the problem completely untouched. For example, graphs with $h(G)=1$ can be recognized in polynomial time due to characterization from \cite{BritnellW13}. However, we do not know if deciding whether $h(G)\leq 2$ can be done in polynomial time.
\item While it is natural to assume that the problem is NP-hard or even PSPACE-hard, we do not have a proof confirming such an assumption.
\item It would be interesting to see if the hunting number of a tree can be computed in polynomial time. 
\end{itemize}

\paragraph{Acknowledgement} The second author is grateful to  Sergiu Hart for interesting discussions on rabbit hunting 
during a   dinner in the Conference Centre ``De Werelt'' in Lunteren, the Netherlands. These discussions served as  the starting point  for our studies.

\def\cprime{$'$}

\end{document}

%% file: Fig4.pdf_t
\begin{picture}(0,0)%
\includegraphics{Fig4.pdf}%
\end{picture}%
\setlength{\unitlength}{3947sp}%
\begingroup\makeatletter\ifx\SetFigFont\undefined%
\gdef\SetFigFont#1#2#3#4#5{%
  \reset@font\fontsize{#1}{#2pt}%
  \fontfamily{#3}\fontseries{#4}\fontshape{#5}%
  \selectfont}%
\fi\endgroup%
\begin{picture}(3048,3486)(242,-2614)
\put(3226,689){\makebox(0,0)[lb]{\smash{{\SetFigFont{12}{14.4}{\rmdefault}{\mddefault}{\updefault}{\color[rgb]{0,0,0}$(10,10)$}%
}}}}
\put(2326,-2536){\makebox(0,0)[lb]{\smash{{\SetFigFont{12}{14.4}{\rmdefault}{\mddefault}{\updefault}{\color[rgb]{0,0,0}$U_8$}%
}}}}
\put(1051,-2536){\makebox(0,0)[lb]{\smash{{\SetFigFont{12}{14.4}{\rmdefault}{\mddefault}{\updefault}{\color[rgb]{0,0,0}$W_2$}%
}}}}
\put(301,689){\makebox(0,0)[lb]{\smash{{\SetFigFont{12}{14.4}{\rmdefault}{\mddefault}{\updefault}{\color[rgb]{0,0,0}$(1,10)$}%
}}}}
\put(301,-2461){\makebox(0,0)[lb]{\smash{{\SetFigFont{12}{14.4}{\rmdefault}{\mddefault}{\updefault}{\color[rgb]{0,0,0}$(1,1)$}%
}}}}
\put(3226,-2386){\makebox(0,0)[lb]{\smash{{\SetFigFont{12}{14.4}{\rmdefault}{\mddefault}{\updefault}{\color[rgb]{0,0,0}$(10,1)$}%
}}}}
\end{picture}%

%% file: Fig5.pdf_t
\begin{picture}(0,0)%
\includegraphics{Fig5.pdf}%
\end{picture}%
\setlength{\unitlength}{3947sp}%
\begingroup\makeatletter\ifx\SetFigFont\undefined%
\gdef\SetFigFont#1#2#3#4#5{%
  \reset@font\fontsize{#1}{#2pt}%
  \fontfamily{#3}\fontseries{#4}\fontshape{#5}%
  \selectfont}%
\fi\endgroup%
\begin{picture}(6046,3147)(383,-2541)
\put(4876,-2461){\makebox(0,0)[lb]{\smash{{\SetFigFont{12}{14.4}{\rmdefault}{\mddefault}{\updefault}{\color[rgb]{0,0,0}b)}%
}}}}
\put(1703,-2468){\makebox(0,0)[lb]{\smash{{\SetFigFont{12}{14.4}{\rmdefault}{\mddefault}{\updefault}{\color[rgb]{0,0,0}a)}%
}}}}
\end{picture}%

%% file: Fig6.pdf_t
\begin{picture}(0,0)%
\includegraphics{Fig6.pdf}%
\end{picture}%
\setlength{\unitlength}{3947sp}%
\begingroup\makeatletter\ifx\SetFigFont\undefined%
\gdef\SetFigFont#1#2#3#4#5{%
  \reset@font\fontsize{#1}{#2pt}%
  \fontfamily{#3}\fontseries{#4}\fontshape{#5}%
  \selectfont}%
\fi\endgroup%
\begin{picture}(6093,3147)(328,-2541)
\put(4876,-2461){\makebox(0,0)[lb]{\smash{{\SetFigFont{12}{14.4}{\rmdefault}{\mddefault}{\updefault}{\color[rgb]{0,0,0}b)}%
}}}}
\put(1703,-2468){\makebox(0,0)[lb]{\smash{{\SetFigFont{12}{14.4}{\rmdefault}{\mddefault}{\updefault}{\color[rgb]{0,0,0}a)}%
}}}}
\end{picture}%

%% file: Fig7.pdf_t
\begin{picture}(0,0)%
\includegraphics{Fig7.pdf}%
\end{picture}%
\setlength{\unitlength}{3947sp}%
\begingroup\makeatletter\ifx\SetFigFont\undefined%
\gdef\SetFigFont#1#2#3#4#5{%
  \reset@font\fontsize{#1}{#2pt}%
  \fontfamily{#3}\fontseries{#4}\fontshape{#5}%
  \selectfont}%
\fi\endgroup%
\begin{picture}(6090,3147)(331,-2541)
\put(4876,-2461){\makebox(0,0)[lb]{\smash{{\SetFigFont{12}{14.4}{\rmdefault}{\mddefault}{\updefault}{\color[rgb]{0,0,0}b)}%
}}}}
\put(1703,-2468){\makebox(0,0)[lb]{\smash{{\SetFigFont{12}{14.4}{\rmdefault}{\mddefault}{\updefault}{\color[rgb]{0,0,0}a)}%
}}}}
\end{picture}%

%% file: Fig8.pdf_t
\begin{picture}(0,0)%
\includegraphics{Fig8.pdf}%
\end{picture}%
\setlength{\unitlength}{3947sp}%
\begingroup\makeatletter\ifx\SetFigFont\undefined%
\gdef\SetFigFont#1#2#3#4#5{%
  \reset@font\fontsize{#1}{#2pt}%
  \fontfamily{#3}\fontseries{#4}\fontshape{#5}%
  \selectfont}%
\fi\endgroup%
\begin{picture}(6090,3147)(331,-2541)
\put(4876,-2461){\makebox(0,0)[lb]{\smash{{\SetFigFont{12}{14.4}{\rmdefault}{\mddefault}{\updefault}{\color[rgb]{0,0,0}b)}%
}}}}
\put(1703,-2468){\makebox(0,0)[lb]{\smash{{\SetFigFont{12}{14.4}{\rmdefault}{\mddefault}{\updefault}{\color[rgb]{0,0,0}a)}%
}}}}
\end{picture}%

%% file: Fig3.pdf_t
\begin{picture}(0,0)%
\includegraphics{Fig3.pdf}%
\end{picture}%
\setlength{\unitlength}{3947sp}%
\begingroup\makeatletter\ifx\SetFigFont\undefined%
\gdef\SetFigFont#1#2#3#4#5{%
  \reset@font\fontsize{#1}{#2pt}%
  \fontfamily{#3}\fontseries{#4}\fontshape{#5}%
  \selectfont}%
\fi\endgroup%
\begin{picture}(7195,6969)(-120,-6456)
\put(5626,-4049){\makebox(0,0)[lb]{\smash{{\SetFigFont{12}{14.4}{\rmdefault}{\mddefault}{\updefault}{\color[rgb]{0,0,0}$i$}%
}}}}
\put(826,-1731){\makebox(0,0)[lb]{\smash{{\SetFigFont{12}{14.4}{\rmdefault}{\mddefault}{\updefault}{\color[rgb]{0,0,0}$i$}%
}}}}
\put(826,-6387){\makebox(0,0)[lb]{\smash{{\SetFigFont{12}{14.4}{\rmdefault}{\mddefault}{\updefault}{\color[rgb]{0,0,0}$i$}%
}}}}
\put(3226,-6374){\makebox(0,0)[lb]{\smash{{\SetFigFont{12}{14.4}{\rmdefault}{\mddefault}{\updefault}{\color[rgb]{0,0,0}$i$}%
}}}}
\put(5176,-4486){\makebox(0,0)[lb]{\smash{{\SetFigFont{12}{14.4}{\rmdefault}{\mddefault}{\updefault}{\color[rgb]{0,0,0}-clear}%
}}}}
\put(5176,-4936){\makebox(0,0)[lb]{\smash{{\SetFigFont{12}{14.4}{\rmdefault}{\mddefault}{\updefault}{\color[rgb]{0,0,0}-contaminated}%
}}}}
\put(5176,-5386){\makebox(0,0)[lb]{\smash{{\SetFigFont{12}{14.4}{\rmdefault}{\mddefault}{\updefault}{\color[rgb]{0,0,0}-shot}%
}}}}
\put(826,-4049){\makebox(0,0)[lb]{\smash{{\SetFigFont{12}{14.4}{\rmdefault}{\mddefault}{\updefault}{\color[rgb]{0,0,0}$i$}%
}}}}
\put(3226,-1724){\makebox(0,0)[lb]{\smash{{\SetFigFont{12}{14.4}{\rmdefault}{\mddefault}{\updefault}{\color[rgb]{0,0,0}$i$}%
}}}}
\put(5626,-1737){\makebox(0,0)[lb]{\smash{{\SetFigFont{12}{14.4}{\rmdefault}{\mddefault}{\updefault}{\color[rgb]{0,0,0}$i$}%
}}}}
\put(3226,-4062){\makebox(0,0)[lb]{\smash{{\SetFigFont{12}{14.4}{\rmdefault}{\mddefault}{\updefault}{\color[rgb]{0,0,0}$i$}%
}}}}
\end{picture}%

%% file: Fig0.pdf_t
\begin{picture}(0,0)%
\includegraphics{Fig0.pdf}%
\end{picture}%
\setlength{\unitlength}{3947sp}%
\begingroup\makeatletter\ifx\SetFigFont\undefined%
\gdef\SetFigFont#1#2#3#4#5{%
  \reset@font\fontsize{#1}{#2pt}%
  \fontfamily{#3}\fontseries{#4}\fontshape{#5}%
  \selectfont}%
\fi\endgroup%
\begin{picture}(2408,1086)(211,-289)
\put(2551,614){\makebox(0,0)[lb]{\smash{{\SetFigFont{12}{14.4}{\rmdefault}{\mddefault}{\updefault}{\color[rgb]{0,0,0}$v_6$}%
}}}}
\put(226,-211){\makebox(0,0)[lb]{\smash{{\SetFigFont{12}{14.4}{\rmdefault}{\mddefault}{\updefault}{\color[rgb]{0,0,0}$v_1$}%
}}}}
\put(1126,-211){\makebox(0,0)[lb]{\smash{{\SetFigFont{12}{14.4}{\rmdefault}{\mddefault}{\updefault}{\color[rgb]{0,0,0}$v_2$}%
}}}}
\put(751,614){\makebox(0,0)[lb]{\smash{{\SetFigFont{12}{14.4}{\rmdefault}{\mddefault}{\updefault}{\color[rgb]{0,0,0}$v_3$}%
}}}}
\put(2026,-211){\makebox(0,0)[lb]{\smash{{\SetFigFont{12}{14.4}{\rmdefault}{\mddefault}{\updefault}{\color[rgb]{0,0,0}$v_4$}%
}}}}
\put(1651,614){\makebox(0,0)[lb]{\smash{{\SetFigFont{12}{14.4}{\rmdefault}{\mddefault}{\updefault}{\color[rgb]{0,0,0}$v_5$}%
}}}}
\end{picture}%

%% file: tk.pdf_tex
%% Creator: Inkscape inkscape 0.48.4, www.inkscape.org
%% Accompanies image file 'tk.pdf' (pdf, eps, ps)
\begingroup%
  \makeatletter%
  \providecommand\color[2][]{%
    \errmessage{(Inkscape) Color is used for the text in Inkscape, but the package 'color.sty' is not loaded}%
    \renewcommand\color[2][]{}%
  }%
  \providecommand\transparent[1]{%
    \errmessage{(Inkscape) Transparency is used (non-zero) for the text in Inkscape, but the package 'transparent.sty' is not loaded}%
    \renewcommand\transparent[1]{}%
  }%
  \providecommand\rotatebox[2]{#2}%
  \ifx\svgwidth\undefined%
    \setlength{\unitlength}{564.58837891bp}%
    \ifx\svgscale\undefined%
      \relax%
    \else%
      \setlength{\unitlength}{\unitlength * \real{\svgscale}}%
    \fi%
  \else%
    \setlength{\unitlength}{\svgwidth}%
  \fi%
  \global\let\svgwidth\undefined%
  \global\let\svgscale\undefined%
  \makeatother%
  \begin{picture}(1,0.37255472)%
    \put(0,0){\includegraphics[width=\unitlength]{tk.pdf}}%
    \put(0.48770403,0.36394225){\color[rgb]{0,0,0}\makebox(0,0)[lb]{\smash{{\tiny{$u$}}}}}%
    \put(0.17016301,0.04209355){\color[rgb]{0,0,0}\makebox(0,0)[lb]{\smash{{\tiny{$Q^{1,p}$}}}}}%
    \put(0.07688779,0.04271745){\color[rgb]{0,0,0}\makebox(0,0)[lb]{\smash{{\tiny{$Q^{1,2}$}}}}}%
    \put(0.02962157,0.04293383){\color[rgb]{0,0,0}\makebox(0,0)[lb]{\smash{{\tiny{$Q^{1,1}$}}}}}%
    \put(0.10618347,0.25092474){\color[rgb]{0,0,0}\makebox(0,0)[lb]{\smash{{\tiny{$v^1$}}}}}%
    \put(0.01861276,0.16128769){\color[rgb]{0,0,0}\makebox(0,0)[lb]{\smash{{\tiny{$w^{1,1}$}}}}}%
    \put(0.10655221,0.16128769){\color[rgb]{0,0,0}\makebox(0,0)[lb]{\smash{{\tiny{$w^{1,2}$}}}}}%
    \put(0.1883431,0.16128769){\color[rgb]{0,0,0}\makebox(0,0)[lb]{\smash{{\tiny{$w^{1,p}$}}}}}%
    \put(0.12465706,0.146288){\color[rgb]{0,0,0}\makebox(0,0)[lb]{\smash{{\tiny{$\ldots$}}}}}%
    \put(0.10292214,0.00235791){\color[rgb]{0,0,0}\makebox(0,0)[lb]{\smash{{\tiny{$P^1$}}}}}%
    \put(0.43918248,0.04209355){\color[rgb]{0,0,0}\makebox(0,0)[lb]{\smash{{\tiny{$Q^{2,p}$}}}}}%
    \put(0.34590725,0.04271745){\color[rgb]{0,0,0}\makebox(0,0)[lb]{\smash{{\tiny{$Q^{2,2}$}}}}}%
    \put(0.29864105,0.04293383){\color[rgb]{0,0,0}\makebox(0,0)[lb]{\smash{{\tiny{$Q^{2,1}$}}}}}%
    \put(0.37520291,0.25092474){\color[rgb]{0,0,0}\makebox(0,0)[lb]{\smash{{\tiny{$v^2$}}}}}%
    \put(0.2876322,0.16128769){\color[rgb]{0,0,0}\makebox(0,0)[lb]{\smash{{\tiny{$w^{2,1}$}}}}}%
    \put(0.37557168,0.16128769){\color[rgb]{0,0,0}\makebox(0,0)[lb]{\smash{{\tiny{$w^{2,2}$}}}}}%
    \put(0.4573626,0.16128769){\color[rgb]{0,0,0}\makebox(0,0)[lb]{\smash{{\tiny{$w^{2,p}$}}}}}%
    \put(0.39367656,0.146288){\color[rgb]{0,0,0}\makebox(0,0)[lb]{\smash{{\tiny{$\ldots$}}}}}%
    \put(0.37194158,0.00235791){\color[rgb]{0,0,0}\makebox(0,0)[lb]{\smash{{\tiny{$P^2$}}}}}%
    \put(0.86339848,0.04209355){\color[rgb]{0,0,0}\makebox(0,0)[lb]{\smash{{\tiny{$Q^{p,p}$}}}}}%
    \put(0.77012325,0.04271745){\color[rgb]{0,0,0}\makebox(0,0)[lb]{\smash{{\tiny{$Q^{p,2}$}}}}}%
    \put(0.72285698,0.04293383){\color[rgb]{0,0,0}\makebox(0,0)[lb]{\smash{{\tiny{$Q^{p,1}$}}}}}%
    \put(0.79941891,0.25092474){\color[rgb]{0,0,0}\makebox(0,0)[lb]{\smash{{\tiny{$v^p$}}}}}%
    \put(0.7118482,0.16128769){\color[rgb]{0,0,0}\makebox(0,0)[lb]{\smash{{\tiny{$w^{p,1}$}}}}}%
    \put(0.79978759,0.16128769){\color[rgb]{0,0,0}\makebox(0,0)[lb]{\smash{{\tiny{$w^{p,2}$}}}}}%
    \put(0.88157855,0.16128769){\color[rgb]{0,0,0}\makebox(0,0)[lb]{\smash{{\tiny{$w^{p,p}$}}}}}%
    \put(0.81789251,0.146288){\color[rgb]{0,0,0}\makebox(0,0)[lb]{\smash{{\tiny{$\ldots$}}}}}%
    \put(0.79615758,0.00235791){\color[rgb]{0,0,0}\makebox(0,0)[lb]{\smash{{\tiny{$P^p$}}}}}%
    \put(0.58695836,0.12430049){\color[rgb]{0,0,0}\makebox(0,0)[lb]{\smash{{\tiny{$\ldots$}}}}}%
  \end{picture}%
\endgroup%

%% file: Rabbit-arxiv.bbl
\begin{thebibliography}{10}

\bibitem{AdlerRSSV03}
{\sc M.~Adler, H.~R{\"{a}}cke, N.~Sivadasan, C.~Sohler, and B.~V{\"{o}}cking},
  {\em Randomized pursuit-evasion in graphs}, Combinatorics, Probability {\&}
  Computing, 12 (2003), pp.~225--244.

\bibitem{BonatoN11}
{\sc A.~Bonato and R.~J. Nowakowski}, {\em The game of cops and robbers on
  graphs}, vol.~61 of Student Mathematical Library, American Mathematical
  Society, Providence, RI, 2011.

\bibitem{Brass09}
{\sc P.~Brass, K.~D. Kim, H.-S. Na, and C.-S. Shin}, {\em Escaping offline
  searchers and isoperimetric theorems}, Comput. Geom., 42 (2009),
  pp.~119--126.

\bibitem{BritnellW13}
{\sc J.~R. Britnell and M.~Wildon}, {\em Finding a princess in a palace: a
  pursuit-evasion problem}, Electron. J. Combin., 20 (2013), pp.~Paper 25, 8.

\bibitem{FominT08}
{\sc F.~V. Fomin and D.~M. Thilikos}, {\em An annotated bibliography on
  guaranteed graph searching}, Theoret. Comput. Sci., 399 (2008), pp.~236--245.

\bibitem{KirousisP85-In}
{\sc L.~M. Kirousis and C.~H. Papadimitriou}, {\em Interval graphs and
  searching}, Discrete Math., 55 (1985), pp.~181--184.

\bibitem{KirousisP86-Se}
\leavevmode\vrule height 2pt depth -1.6pt width 23pt, {\em Searching and
  pebbling}, Theoret. Comput. Sci., 47 (1986), pp.~205--218.

\bibitem{NowakowskiW83}
{\sc R.~Nowakowski and P.~Winkler}, {\em Vertex-to-vertex pursuit in a graph},
  Discrete Math., 43 (1983), pp.~235--239.

\bibitem{Parsons78}
{\sc T.~D. Parsons}, {\em Pursuit-evasion in a graph}, in Theory and
  applications of graphs ({P}roc. {I}nternat. {C}onf., {W}estern {M}ich.
  {U}niv., {K}alamazoo, {M}ich., 1976), Springer, Berlin, 1978, pp.~426--441.
  Lecture Notes in Math., Vol. 642.

\bibitem{Petrov82}
{\sc N.~N. Petrov}, {\em A problem of pursuit in the absence of information on
  the pursued}, Differentsial\cprime nye Uravneniya, 18 (1982), pp.~1345--1352,
  1468.

\bibitem{Quilliot85}
{\sc A.~Quilliot}, {\em A short note about pursuit games played on a graph with
  a given genus.}, J. Comb. Theory, Ser. B, 38 (1985), pp.~89--92.

\bibitem{Tosic85}
{\sc R.~To{\v{s}}i{\'c}}, {\em Vertex-to-vertex search in a graph}, in Graph
  theory (Dubrovnik 1985), 1985, pp.~233--23.

\end{thebibliography}
